\documentclass[a4paper, 12pt, preprint, 3p, times]{elsarticle}
\usepackage{amssymb}
\usepackage{epsfig}
\usepackage{amsfonts}
\usepackage{amsmath}
\usepackage{euscript}
\usepackage{amscd}
\usepackage{amsthm}
\DeclareMathAlphabet{\mathpzc}{OT1}{pzc}{m}{it}
\newtheorem{thm}{Theorem}[section]
\newtheorem{lem}[thm]{Lemma}
\newtheorem{prop}[thm]{Proposition}
\newtheorem{prob}[thm]{Problem}
\newtheorem{cor}[thm]{Corollary}

\newdefinition{defn}[thm]{Definition}
\newdefinition{ex}[thm]{Example}
\newdefinition{rem}[thm]{Remark}
\newdefinition{note}{Note}
\newdefinition{q}{Question}

\newcommand{\comment}[1]{}

\newcommand\p {\mathfrak{p}}

\newcommand{\A}[1]{\mathbb{A}^{#1}}
\newcommand{\ol}[1] {\overline{#1}}

\newcommand{\nothookrightarrow}{ \hookrightarrow \hspace*{-5mm} / \hspace*{2.5mm}}

\begin{document}
\begin{frontmatter}
\title{On cancellation of variables of the form $bT^n-a$\\ over affine normal domains}

\author{Prosenjit Das}
\address{Department of Mathematics, Indian Institute of Space Science and Technology, \\
Valiamala P.O., Trivandrum 695 547, India\\
email: \texttt{prosenjit.das@gmail.com, prosenjit.das@iist.ac.in}}
%\author{Amartya K. Dutta}
%\address{Stat-Math Unit, Indian Statistical Institute,\\
%203, B.T. Road, Kolkata 700 108, India\\
%email: \texttt{amartya@isical.ac.in}}
%
\begin{abstract}
In this article we extend a cancellation theorem of D. Wright to the case of affine normal domains. We shall show that if $A$ is an algebra over a Noetherian normal domain $R$ containing a field $k$ and if  $A[T] = R^{[3]}$, then $A = R^{[2]}$ if and only if $A[T]$ has a variable of the form $bT^n-a$ for some $a, b \in A$ with $n \ge 2$ and ch$(k) \nmid n$.

\end{abstract}
\end{frontmatter}

\section{Introduction} \label{sec_intro}
Throughout the article rings will be commutative with unity. For a ring $R$, $R^{[n]}$ will denote the polynomial ring in $n$ variables over $R$. We shall use the notation $A = R^{[n]}$ to mean that $A$ is isomorphic, as an $R$-algebra, to a polynomial ring in $n$ variables over $R$. For a prime ideal $P$ of $R$, $k(P)$ will denote the residue field $R_P/PR_P$; and for an $R$-algebra $A$, $A_P$ will denote the localisation $S^{-1}A$ where $S = R \backslash P$.\\%$K$ will denote the field of fractions of an integral domain $R$; and by $(R,\pi)$ we shall denote a discrete valuation ring with uniformizing parameter $\pi$.\\

Consider the following cancellation problem.

\begin{prob} \label{can_prob}
Let $R$ be a ring, $A$ an $R$-algebra and $A[T] = A^{[1]}$. Suppose, $A[T] = R^{[3]}$. Is then $A = R^{[2]}$?
\end{prob}

While the problem is open in general, it is well known that the contributions of Miyanishi-Sugie (\cite{Miyanishi_cyllinder}) and Fujita (\cite{Fujita_Zariski-Problem}) give affirmative solution to the problem for the case $R$ is a field of characteristic zero; then Russell (\cite{Russell_Cancellation}) gave affirmative solution when $R$ is a perfect field of arbitrary characteristic; and recently Bhatwadekar-Gupta \cite{Neena-Bhatwadekar_Epi} showed that the same holds even when $R$ is a non-perfect field. When $R$ is PID containing a field of characteristic zero, the results on $\mathbb{A}^2$-fibration by Sathaye (\cite{Sat_Pol-two-var-DVR}), along with a result of Bass-Connell-Wright (\cite{BCR_Local-Poly}), show that $A$, indeed, is a polynomial ring in two variables. Asanuma-Bhatwadekar's structure theorem on $\mathbb{A}^2$-fibration shows the same conclusion when $R$ is a one dimensional Noetherian local domain containing $\mathbb{Q}$ such that $\Omega_R(A)$ is extended from $R$ (\cite{Asan-Bhatw_Struct-A2-fib}, Corollary 3.9); in particular, Problem \ref{can_prob} has an affirmative solution when $R$ is a one dimensional Noetherian seminormal local domain containing $\mathbb{Q}$ (\cite{Asan-Bhatw_Struct-A2-fib}, Remark 3.10); also see (\cite{Derksen_cancellation} and \cite{Kahoui_Cancellation}). However, even when $R$ is a PID (but $\mathbb{Q} \nothookrightarrow R$), Problem \ref{can_prob} does not have an affirmative answer by an example of Asanuma (\cite{Asanuma_fibre_ring}, Theorem 5.1). We present below a generalised version of the example due to Neena Gupta (\cite{Neena_Threefold}, Lemma 3.2, Theorem  4.2, Theorem 4.3).
\begin{ex} \label{asa_example}
Let $k$ be a field of characteristic $p \ne 0$ and $R=k[\pi] = k^{[1]}$. Set $\displaystyle A := \frac{R[X,Y,Z]}{(\pi^m Z -f(X,Y))}$ where $m$ is a positive integer and $f(X,Y) \in k[X,Y] = k^{[2]}$ be such that $k[X,Y]/(f(X,Y)) = k[\overline{h(X,Y)}] = k^{[1]}$, $\overline{h(X,Y)}$ being image of some $h(X,Y) \in k[X,Y]$ and $k[X,Y] \ne k[f(X,Y)]^{[1]}$. Then $A \otimes_R k(P) = k(P)^{[2]}$ for all $P \in$ Spec$(R)$ and $A[T] = R[\pi^m T - h(X,Y)]^{[2]} = R^{[3]}$, but $A \ne R^{[2]}$.
\end{ex}
 
Naturally, one may ask that under what conditions a positive answer to Problem \ref{can_prob} can be expected over general rings. It is important to note that if there exists an element $F$ in $A[T]\backslash A$ satisfying $A[T]= R[F]^{[2]}$ so that $B:=A[T]/(F) = R^{[2]}$ becomes a simple extension of $A$, then one may try to construct variables of $A$ from judiciously chosen variables of $B$. So, corresponding to Problem \ref{can_prob}, the following epimorphism problem can be considered.

\begin{prob} \label{can-epi_prob}
Let $A$ be a finitely generated algebra over a ring $R$ and $A[T] = A^{[1]}$. Suppose, there exists $F \in A[T] \backslash A$ such that $B := A[T] / (F) = R^{[2]}$. Then 
\begin{enumerate}
\item [\rm (i)] Is $A = R^{[2]}$?
\item [\rm (ii)] Is $A[T] = R[F]^{[2]}$? 
\end{enumerate}

\end{prob}

When $R$ is a field and $F = bT^n-a$, where $a, b \in A$, positive answers have been given by Peter Russell \cite{R_BIR} and David Wright \cite{DW_CNCL} for the case $n=1$ and $n \ge 2$ respectively, under certain assumptions on $A$ (also see \cite{Sathaye_OLP} and \cite{R_GAL}). We quote below the precise statement of D. Wright.

\begin{thm}\label{david_cncl}
Let $k$ be an algebraically closed field of characteristic $p \ge 0$ and $A$ a normal affine $k$-domain. Let $a,b \in A$, $b \ne 0$, and suppose that $B = A[T]/(bT^n -a) = k^{[2]}$, where $n \ge 2$ is an integer, not divisible by $p$. Then there are variables $X, Y$ of $B$ such that $Y$ is the image of $T$ in $B$, $b \in k[X]$, $a=bY^n$, and $A=k[X,a]=k^{[2]}$. Moreover, $A[T] = k[X, bT^n-a, T] = k[bT^n-a]^{[2]}$.
\end{thm}

Under the hypothesis $A = R^{[2]}$, Das-Dutta in \cite{DD_plane} showed that Wright's epimorphism result extends to more general rings, thereby answering (ii) of Problem \ref{can-epi_prob} in the case of such rings for $F = bT^n-a$. We quote below one of the results.

\begin{thm} \label{DD_plane-thm-F}
Let $R$ be a Noetherian normal domain containing a field of characteristic $p \ge 0$, $A = R[X,Y] = R^{[2]}$ and $F \in A[T](=R[X,Y,T]=R^{[3]})$ be of the form $bT^n-a$ where $a,b \in R[X, Y]$, $b \ne 0$ and $n$ is an integer $ \ge 2$ with $p \nmid n$. Suppose that $R[X,Y,T]/(F) = R^{[2]}$. Then $R[X, Y, T] = R[F, T]^{[1]}$ and $R[X, Y] = R[a]^{[1]}$.
\end{thm}
In this article, we shall use some recent results on residual variables of affine fibrations by Das-Dutta (\cite{DD_residual}) to show that the above epimorphism result can be generalized to the case $A$ is an $\A{2}$-fibration over $R$ with $\Omega_R (A)$ a stably free $A$-module, thereby getting a partial answer to (i) of Problem \ref{can-epi_prob} for the above mentioned $R$ and $F$. More generally, we shall show the following (Proposition \ref{our_UFD-normal-th}):

\medskip

\noindent
{\bf Proposition A.}
Let $R$ be a Noetherian normal domain containing a field of characteristic $p \ge 0$ and $A$ a finitely generated flat $R$-domain with $\Omega_R(A)$ a stably free $A$-module. Suppose there exist $a,b \in A$ satisfying $\displaystyle{\frac{A[T]}{(bT^n-a)} = R^{[2]}}$, where $n \ge 2$ and $p \nmid n$, and that, for each prime ideal $P$ of $R$, $A \otimes_R k(P)$ is normal and $b \notin PA_P$. Then $A = R[a]^{[1]} = R^{[2]}$ and $A[T] = R[bT^n-a, T]^{[1]} = R^{[3]}$. When $R$ is a factorial domain, the hypothesis ``$\Omega_R(A)$ is stably free'' may be dropped.\\

We shall also see that Problem \ref{can_prob} has an affirmative answer over a Noetherian normal domain $R$ if $A[T]$ has a variable of the form $bT^n-a$; more precisely (Theorem \ref{our_cor_UFD_cancel}):

\medskip

\noindent
{\bf Theorem B.}
Let $R$ be a Noetherian normal domain containing a field of characteristic $p \ge 0$ and $A$ an $R$-algebra. Suppose there exist $a,b \in A$ such that $A[T] = R[bT^n-a]^{[2]} = R^{[3]}$, where $n \ge 2$ and $p \nmid n$. Then $A = R[a]^{[1]} = R^{[2]}$ and $A[T] = R[bT^n-a, T]^{[1]}$.\\

An important question on affine fibration is whether every $\A{2}$-fibration is a polynomial ring in two variables over the base ring. In \cite{Sat_Pol-two-var-DVR}, A. Sathaye showed that an $\A{2}$-fibration $A$ over a base ring $R$ is trivial if $R$ is a DVR containing $\mathbb{Q}$. Asanuma's example (\cite{Asanuma_fibre_ring}, Theorem 5.1) shows that non-trivial $\A{2}$-fibrations exist over a DVR containing a field of positive characteristic. But it is not known whether every $\A{2}$-fibration over a two dimensional regular affine spot containing $\mathbb{Q}$ is trivial. In this article we shall observe that an $\A{2}$-fibration $A$ over a Noetherian factorial domain containing $\mathbb{Q}$ is trivial if there exist $a$, $b$ in $A$ such that the fibres of $\displaystyle \frac{A[T]}{(bT^n-a)}$ are $\A{2}$ (see Corollary \ref{A2-fib-cor}).

\medskip

\noindent
{\bf Corollary C.}
Let $R$ be a Noetherian factorial domain containing $\mathbb{Q}$ and $A$ an $\A{2}$-fibration over $R$. Let $n \ge 2$. Suppose there exist $a,b \in A$ such that $\displaystyle B := \frac{A[T]}{(bT^n-a)}$ satisfies $ B \otimes_R k(P) = k(P)^{[2]}$ for all prime ideals $P$ of $R$. Then $A = R[a]^{[1]} = R^{[2]}$ and $A[T] = R[bT^n-a, T]^{[1]}$.

\section{Preliminaries} \label{sec_prelim}
\begin{defn}
Let $R$ be a ring. An $R$-domain $A$ is called residually normal (factorial) if $A \otimes_R k(P)$ is normal (factorial) for all $P \in$ Spec($R$); a finitely generated flat $R$-algebra $A$ is said to be an $\A{n}$-fibration over $R$ if $A \otimes_R k(P) = k(P)^{[n]}$ for all $P \in$ Spec($R$); an $m$-tuple of algebraically independent elements $(W_1, W_2, \cdots, W_m )$ from an $\A{n}$-fibration $A$ over $R$ is called an $m$-tuple residual variable of $A$ over $R$ if $A \otimes_R k(P) = (R[W_1, W_2, \cdots, W_m ] \otimes_R k(P))^{[n-m]}$ for all $P \in$ Spec($R$).
\end{defn}

The result below is a special case of (\cite{D_SEP}, Theorem 7).
\begin{thm} \label{dut_sep-a1}
Let $k$ be a field, $L$ a separable field extension of $k$, $A$ a factorial domain containing $k$ and $B$ an $A$-algebra such that $B \otimes_k L = (A \otimes_k L)^{[1]}$. Then $B=A^{[1]}$.
\end{thm}

The following result gives a criterion of equality of a ring and its subring (\cite{BD_DVR}, Lemma 2.1):
\begin{lem} \label{Lem_ring-equality}
Let $A \hookrightarrow B$ be domains and $f \in A\backslash \{0\}$ be such that $A[1/f] = B[1/f]$ and $fB \cap A = fA$, then $A=B$. 
\end{lem}
 
We register the following lemma by Sathaye (\cite{Sathaye_OLP}, Lemma 1):
\begin{lem} \label{sat_lem-var}
Let $k$ be a field and suppose $X'$ is a variable in 
$k[X_1, X_2, \dots, X_n] (= k^{[n]})$. Then $X$ is comaximal with $X_1$ if and only if $X' = \alpha X_1 + \beta$ for some $\alpha, \beta \in k^*$.
\end{lem}

The following result by Das-Dutta (\cite{DD_plane}, Lemma 4.1) will be used to prove one of our main results:
\begin{lem} \label{our_auto-lem}
Let $k$ be a field of characteristic $p \ge 0$ and $\sigma$ a $k$-automorphism of $B = k^{[2]}$ of order $n$ such that $p \nmid n$. Suppose that $k$ contains all the $n^{th}$ roots of unity. Then there exist elements $U, V \in B$ and $\alpha, \beta \in k^*$ such that $B = k[U, V]$, $\sigma (U)= \alpha U$ and $\sigma (V) =\beta V$, where $\alpha^n = \beta^n  = 1$.
\end{lem}
We shall use the following consequence of Sathaye's result (\cite{Sathaye_OLP}, Corollary 1) which appears in (\cite{DD_plane}, Lemma 4.2):
\begin{lem} \label{our_rs-lem}
 Let $k$ be a field, $B = k^{[2]}$ and $b \in B \backslash k$. Suppose that there exist a separable algebraic extension $E|_k$ and an element $X' \in B \otimes_k E$ such that $B \otimes_k E = E[X']^{[1]}$ and $b \in E[X']$. Then there exists $X \in B$ such that $b \in k[X]$, $B = k[X]^{[1]}$ and $E[X']=E[X]$.
\end{lem}

Finally, we record a result by Das-Dutta on residual variables (\cite{DD_residual}, Corollary 3.8, Corollary 3.19)

\begin{thm} \label{DD_res_lem-1}
Let $R$ be a Noetherian domain and $A$ an $\A{m+1}$-fibration over $R$ satisfying any one of the following conditions

\begin{enumerate}
\item [\rm (i)] $R$ is factorial domain.
\item [\rm (ii)] $\Omega_R(A)$ is a stably free $A$-module where either $R$ contains $\mathbb{Q}$ or $R$ is seminormal. 
\end{enumerate}
Then an $m$-tuple $(W_1, W_2, \cdots, W_m)$ of $A$ is an $m$-tuple residual variable of $A$ over $R$ if and only if it is an $m$-tuple variable of $A$ over $R$, i.e., $A \otimes_R k(P) = (R[W_1, W_2, \cdots, W_m] \otimes_R k(P))^{[1]}$ for all prime ideals $P$ of $R$ if and only if $A = R[W_1, W_2, \cdots, W_m]^{[1]} = R^{[m+1]}$.
\end{thm}

\section {Main Results}

\subsection*{\bf \S \ Cancelling variables of the form $bT^n -a$ over a field}

We shall first show that Theorem \ref{david_cncl} can be generalised to any field. The proof follows from the proofs of Proposition 4.4 and Theorem 4.5 of \cite{DD_plane}; but for reader's convenience the proof is being included. 

\begin{thm} \label{our_field-th}
Let $k$ be a field of characteristic $p \ge 0$ and $A$ a normal affine $k$-domain. Suppose there exist $a, b \in A$, $b \ne 0$, such that $\displaystyle{B :=\frac{A[T]}{(bT^n -a)} = k^{[2]}}$, where $n \ge 2$ and $p \nmid n$. Then there exist variables $X, Y$ in $B$ such that $Y$ is the image of $T$ in $B$, $b \in k[X]$, $A = k[X, a] = k^{[2]}$ and $A[T] = k[X, bT^n-a, T]$.
\end{thm}
\begin{proof}
\noindent
\underline{Case - I}: Suppose that $k$ contains all the $n^{th}$ roots of unity.

Let $\sigma$ be the $k$-automorphism of $B$ induced by the $k$-automorphism  $\tilde{\sigma}$ of $A[T]$ defined by $\tilde{\sigma}(T) = \omega T$ where $\omega$ is a primitive $n^{th}$ root of unit. Obviously, $\sigma$ has order $n$.

\medskip

Since $B = k^{[2]}$, by Lemma \ref{our_auto-lem} there exist variables $U', V' \in B$ and $\alpha, \ \beta \in k^*$ such that $\sigma(U')= \alpha U'$ and $\sigma(V') = \beta V'$ where $\alpha^n = \beta^n = 1$. Let $t$ be the image of $T$ in $B$ and $C = A[a/b]$. Then $t^n = a/b$ and $B = A[t] = C[t] = C \oplus t C \oplus t^2 C \oplus \cdots \oplus t^{n-1} C$. Observe that for all $x \in B, \ t \mid (x - \sigma(x))$. Thus $t \mid (1-\alpha) U'$ and $t \mid (1-\beta)V'$. Without loss of generality we may assume that $\alpha = 1$ and hence we get that $V'$ is a unit multiple of $t$ and the ring of invariant of $\sigma$ is $C = A[a/b] = k[U', a/b] = k^{[2]}$. 

\medskip

Set $Y :=t$. We shall show that there exists $X \in A$ such that $B = k[X,Y]$, $b \in k[X]$ and $A = k[X,a] = k^{[2]}$.

\medskip

If $b \in k^*$, then $A = A[a/b] = k[U', a/b]$. Then setting $X := U'$ we have $A = k[X,a]$, $B = k[X,Y]$ and $C = k[X, a/b]$. Now, let $b \notin k^*$. Suppose $p_1, p_2, \cdots , p_m$ be distinct irreducible factors of $b$ in $C = k^{[2]}$. We shall show that $p_i$'s are pair wise comaximal. 

\medskip

Let $\p_i = A \cap  p_i C$. Since $a, b \in A \cap  b C \subseteq \p_i$, ht$(\p_i) > 1$ for all $i = 1,2, \cdots , m$. Since dim$(A) = 2$, each $\p_i$ is a maximal ideal of $A$. Let $\bar{k}$ denote an algebraic closure of $k$, $L_i$ be a subfield of $\bar{k}$ isomorphic to $A/\p_i$ and let $L$ be  
the subfield of $\bar{k}$ generated by the fields $L_1, L_2, \dots, L_m$. 
 Then $L_i$ is an algebraic extension of $k$ and  $C/\p_i C = (A/\p_i) [\zeta_i] = L_i[\zeta_i]$ where $\zeta_i$ is the image of $a/b$ in $C/\p_i C$. Since $\p_iC \subseteq p_i C$, it follows that $\zeta_i$ is transcendental over $L_i$ (otherwise $p_i \mid a$ and hence $p_i$ is a non-zero divisor in $C$ which is an integral domain, a contradiction) and $\p_iC$ is a prime ideal of $C$. As ht$\ p_iC =1$ and $\p_i C \ne 0$, we have $p_i C = \p_i C$. This shows that $p_i$ are pairwise comaximal in $C$ and hence in $B$.

\medskip

Let $g(\zeta_i)$ be the image of $U'$ in $C/p_i C = L_i[\zeta_i]$. 
Then $U' -g(a/b)$ is divisible by $p_i$ in $C \otimes_k L_i$. But $U' - g(a/b) = U'-g(Y^n)$ is a variable in both $C\otimes_k L_i$ and $B \otimes_k L_i$. Hence $U' - g(a/b)$ is a constant multiple of $p_i$. Thus $C \otimes_k L_i=L_i[p_i, a/b]$, $B \otimes_k L_i = L_i[p_i, Y]$, and for $i \ne j$, $(p_i, p_j)B \otimes_k L = B \otimes_k L$. Set $X:=p_1$. Using Lemma \ref{sat_lem-var}, we have $p_i = \lambda_i X + \pi_i$ for $\lambda_i \in L^*$ and $\pi_i \in L$. So, we have $b \in L[X]$. This shows that $X$ is integral over $A \otimes_k L$ and hence over $A$. As $X \in A[a/b]$ and $A$ is a normal domain, we have $X \in A$. Since $L|_k$ is faithfully flat, it follows that $B = k[X, Y]$ with $X \in A, Y = t$ and $b \in k[X]$; and $C = k[X, a/b]$.

\medskip

%Now, repeating the arguments in (\cite{DW_CNCL}, pg. 98) we get $A = k[X,a]$. This completes the proof.
Now, to complete the proof, we are only left to show that $A = k[X,a]$. First we claim that $bA \cap k[X,a] = b k[X, a]$. We repeat the argument in (\cite{DW_CNCL}, pg. 98) to prove this claim. Let $h \in bA \cap k[X,a]$. Then
$$
h = h_0(X) + h_1(X)a + \cdots + h_d(X)a^d.
$$
Since $a \in bC$, it follows that $h_0(X) \in bC$. But since $C = k[X, a/b](=k^{[2]})$, we get $h_0(X) \in bk[X]$ and hence $h_0(X) \in bk[X,a]$. So, we may replace $h$ by $h-h_0(X) = h_1(X)a + h_2(X)a^2 + \cdots + h_d(X)a^d$. Let $h' = h_1(X) + h_2(X)a + \cdots + h_d(X)a^{d-1}$. Then $h = h'a$. Since there is no height one prime ideal of $A$ which contains both $a$ and $b$, and since $h'a \in bA$, it follows from the normality of $A$ that (the associative prime ideals of $a$ are of height one) $\displaystyle h'/b \in \underset{\p \in \ \text{Spec}(A); \ ht(\p)=1} {\overset{} {\bigcap}} A_{\p} = A$. Therefore, $h' \in bA$. Now we argue as before that $h_1(X) \in bk[X,a]$. We continue this process to conclude that $h_i(X) \in bk[X,a]$ for $0\le i \le d$, which proves the claim. Now, since $b \in k[X,a]$ is a non-zero element, applying Lemma \ref{Lem_ring-equality} we have $A = k[X,a]= k^{[2]}$.

\medskip

Thus, if $k$ contains all the $n^{th}$ roots of unity, then there exist variables $X, Y$ in $B$ such that $Y$ is the image of $T$ in $B$, $b \in k[X]$, $A = k[X, a] = k^{[2]}$ and $A[T] = k[X, bT^n-a, T]$.

\medskip

Now we take the other case.

\noindent
\underline{Case - II}: Suppose $k$ does not contain all the $n ^{th}$ roots of unity.

Let $E$ be the field obtained by adjoining all the $n^{th}$ roots of unity to $k$ and let $g = bT^n-a$. Since $p \nmid n$, $E$ is a Galois extension over $k$. By Case - I, we get variables $X'$ and $Y'$ of $B \otimes_k E$ ($=(A \otimes_k E)[T]/(g) = E^{[2]}$)  such that $Y'$ is the image of $T$,  $b \in E[X']$ and $A \otimes_k E = E[X',a]$. As $E|_k$ is separable, we have $A = {k[a]}^{[1]}$ by Theorem \ref{dut_sep-a1}. If $b \in A \backslash k$, then, by Lemma \ref{our_rs-lem}, we get $X \in A = {k[a]}^{[1]} = k^{[2]}$ such that $A = k[X]^{[1]}$, $b \in k[X]$ and $E[X] = E[X']$. Since $E|_k$ is faithfully flat, $E[X',a] = E[X, a]$ and $k[X,a] \subseteq A$, we have $A=k[X, a]$. If $b \in k$, then we choose $X$ to be any complementary variable of $a$ in  $A$. Now $A[T] = k[X,a, T] = k[X, bT^n-a, T]$, as $b \in k[X]$. This completes the proof.
\end{proof}

\subsection*{\bf \S \ Cancelling variables of the form $bT^n -a$ over a normal domain}
We now prove Proposition A.

\begin{prop} \label{our_UFD-normal-th}
 Let $R$ be a Noetherian domain containing a field of characteristic $p \ge 0$ and $A$ a finitely generated flat residually normal $R$-domain satisfying any of the following conditions
 
\begin{enumerate}
\item [\rm (i)] $R$ is factorial.
\item [\rm (ii)] $\Omega_R(A)$ is a stably free $A$-module with either $R$ contains $\mathbb{Q}$ or $R$ is seminormal.
\end{enumerate}
 
Let $n \ge 2$ be such that $p \nmid n$. Suppose there exist $a,b \in A$ such that, for each $P \in$ Spec($R$), $b \notin PA_P$ and $\displaystyle B :=\frac{A[T]}{(bT^n-a)}$ satisfies $B \otimes_R k(P) = k(P)^{[2]}$. Then $A = R[a]^{[1]} = R^{[2]}$ and $A[T] = R[bT^n-a, T]^{[1]}$.
 \end{prop}
\begin{proof}
Fix $P \in$ Spec($R$). Letting $\ol{a}$ and $\ol{b}$ respectively denote the images of $a$ and $b$ in $A \otimes_R k(P)$, we have $\displaystyle{B \otimes_R k(P) = \frac{(A \otimes_R k(P))[T]}{(\ol{b}T^n-\ol{a})} = k(P)^{[2]}}$. Then since $A \otimes_R k(P)$ is a normal affine $k(P)$-domain, by Theorem \ref{our_field-th}, there exist variables $X_P, Y_P$ in $B \otimes_R k(P)$ such that $Y_P$ is the image of $T$ in $B \otimes_R k(P)$, $X_P \in A \otimes_R k(P)$, $A\otimes_R k(P) = k(P)[X_P,\ol{a}] = k(P)^{[2]}$ and $A[T]\otimes_R k(P) = k(P)[T,\ol{b}T^n -\ol{a}]^{[1]} = k(P)^{[3]}$. 

\medskip

This shows that $a$ is a residual variable of $A$ over $R$ and $(bT^n-a, T)$ is a pair of residual variables of $A[T]$ over $R$. Since $R$ is a Noetherian domain, and since $A$ is finitely generated flat $R$-algebra, by Theorem \ref{DD_res_lem-1} we have $A = R[a]^{[1]} = R^{[2]}$ and $A[T] = R[bT^n-a, T]^{[1]}$, if either $\Omega_R(A)$ is a stably free module or $R$ is a factorial domain. This completes the proof.
\end{proof}

\begin{rem}
\item [\rm (I)] In Proposition \ref{our_UFD-normal-th}, if we assume that $R$ is a factorial domain, then it can be seen that there exists $X' \in A$ such that $b \in R[X']$ and $A \otimes_R K = K[X', a]$.

\medskip

\item [\rm (II)] When $R$ is a DVR, the following holds as a special case of Proposition \ref{our_UFD-normal-th}:

\textit{Let $(R, \pi)$ be a DVR containing a field of characteristic $p \ge 0$ and $A$ a finitely generated residually normal $R$-domain. Suppose there exist $a, b \in A$, $\pi \nmid b$, such that $\displaystyle{\frac{A[T]}{(bT^n-a)} = R^{[2]}}$, where $n \ge 2$ and $p \nmid n$. Then there exists $X' \in A$ such that $b \in R[X']$, $A[1/\pi] = R[1/\pi][X', a]$, $A = R[a]^{[1]} = R^{[2]}$ and $A[T] = R[bT^n-a, T]^{[1]}$.}

\medskip

\item [\rm (III)]In Proposition \ref{our_UFD-normal-th}, if we assume that $b=1$ (or $b \in R^*$), then the condition ``$A$ is a residually normal domain'' holds automatically due to the fact that $\displaystyle B \otimes_R k(P) = \frac{(A \otimes_R k(P))[T]}{(T^n-\bar{a})} = k(P)^{[2]}$ is a normal domain and is a free $A \otimes_R k(P)$-module.

\item [\rm (IV)] When $A = R^{[2]}$, the conclusion of Proposition \ref{our_UFD-normal-th} follows even if $b$ belongs to $PA_P$ for some prime ideal $P$ of $R$ (see \cite{DD_plane}, Theorem 6.2).
\end{rem}
 
As another consequence of Theorem \ref{our_field-th} we see that the answer to Problem \ref{can_prob} is affirmative over Noetherian normal domains, if $A[T]$ has a variable of the form $bT^n-a$ where $n \ge 2$.

\begin{thm} \label{our_cor_UFD_cancel}
Let $R$ be a Noetherian domain containing a field of characteristic $p \ge 0$, which either contains $\mathbb{Q}$ or is seminormal; and $A$ an $R$-algebra. Let $a,b \in A$ be such that
$$A[T] = R[bT^n-a]^{[2]} = R^{[3]},$$
where $n \ge 2$ and $p \nmid n$. Then $A = R[a]^{[1]} = R^{[2]}$ and $A[T] = R[bT^n-a, T]^{[1]}$.
\end{thm}

\begin{proof}
Note that since $A[T] = R^{[3]}$, $A$ is a finitely generated flat residually factorial $R$-domain; and by (\cite{DD_residual}, Lemma 2.1), $\Omega_R(A)$ is a stably free $A$-module.

\medskip

Fix a prime ideal $P$ of $R$. If $b \notin PA_P$, then applying Theorem \ref{our_field-th}, we have $A \otimes_R k(P) = (R[a] \otimes_R k(P))^{[1]}$ and $A[T] \otimes_R k(P) = (R[bT^n-a, T] \otimes_R k(P))^{[1]}$. If $b \in PA_P$, then $(A \otimes_R k(P))^{[1]} = (R[bT^n -a] \otimes_R k(P))^{[2]} = (R[a] \otimes_R k(P))^{[2]}$. By a result of Abhyankar-Eakin-Heinzer (\cite{AEH_Coff}, 4.5) we have $(A \otimes_R k(P)) =  (R[a] \otimes_R k(P))^{[1]}$, and hence $A[T] \otimes_R k(P) = k(P)[\bar{a},T]^{[1]} = k(P)[\bar{b}T^n-\bar{a}, T]^{[1]} = (R[bT^n-a, T] \otimes_R k(P))^{[1]}$. Since $P \in$ Spec($R$) is arbitrary, using Theorem \ref{DD_res_lem-1}, we get $A = R[a]^{[1]}=R^{[2]}$ and $A[T] = R[bT^n-a, T]^{[1]}$. This completes the proof.
\end{proof}

\begin{rem}

\item [\rm (I)] The converse of Theorem \ref{our_cor_UFD_cancel} holds: If $A = R^{[2]}$, then there exist $a,b \in A$, e.g., $a = Y$ and $b = X$, such that $A[T] = R[bT^n-a]^{[2]}$.

\item [\rm (II)] Example \ref{asa_example} shows the necessity of the condition $n \ge 2$ in Theorem \ref{our_cor_UFD_cancel}.
\end{rem}

Theorem \ref{our_cor_UFD_cancel} shows that in Example \ref{asa_example}, $A[T]$ does not have any coordinate plane of the form $bT^n-a$ with $a$, $b$ in $A$, $n \ge 2$ and $p \nmid n$. More generally, we have

\begin{cor} \label{Our_Cor_Neena}
Let $k$ be a field of characteristic $p \ge 0$, $R = k[\pi] = k^{[1]}$ and $\displaystyle A = \frac{R[X,Y,Z]}{(\pi^m Z - F(\pi, X,Y))}$ where $m \ge 1$ and $\displaystyle \frac{k[X,Y]}{(F(0, X,Y))} = k^{[1]}$. Then $A = R^{[2]}$ if and only if there exist $a,b \in A$ and $n \ge 2$ with $p \nmid n$ such that for $\displaystyle B := \frac{A[T]}{(bT^n-a)}$ we have

\bigskip

either

\item [\rm (I)] for each $P \in$ Spec$(R)$, $B \otimes_R k(P) = k(P)^{[2]}$ and $b \notin PA_P$.

\medskip

or

\item [\rm (II)] $B = R^{[2]}$ and $b \notin PA_P$ for each $P \in$ Spec$(R)$.

\medskip

or

\item [\rm (III)] $\displaystyle A[T] = R[bT^n-a]^{[2]}$.
\end{cor}

\begin{proof}
Follows from Proposition \ref{our_UFD-normal-th} \& (\cite{Neena_Threefold}, Lemma 3.2) and Theorem \ref{our_cor_UFD_cancel} \& (\cite{Neena_Threefold}, Theorem 4.2).
\end{proof}

In \cite{Neena_Threefold}, it has been observed that under the hypothesis  $\displaystyle \frac{k[X,Y]}{(F(0, X,Y))} = k^{[1]}$ the algebra $A$ in Corollary \ref{Our_Cor_Neena} is an $\mathbb{A}^2$-fibration over $R$ and vice-versa. Therefore, Corollary \ref{Our_Cor_Neena} states that a certain class of $\mathbb{A}^2$-fibration $A$ is trivial if the fibres of $B$ are $\A{2}$. From Theorem \ref{our_field-th}, we observe below that this phenomenon is true for any $\A{2}$-fibration over a Noetherian factorial domain containing $\mathbb{Q}$.

\begin{cor} \label{A2-fib-cor}
 Let $R$ be a Noetherian domain containing $\mathbb{Q}$ and $A$ an $\A{2}$-fibration over $R$ such that either $R$ is factorial or $\Omega_R(A)$ is a stably free $A$-module. Let $n \ge 2$. Suppose there exist $a,b \in A$ such that $\displaystyle B : = \frac{A[T]}{(bT^n-a)}$ satisfies $B \otimes_R k(P) = k(P)^{[2]}$ for all $P \in$ Spec$(R)$. Then $A = R[a]^{[1]} = R^{[2]}$ and $A[T] = R[bT^n-a, T]^{[1]}$.
 
\end{cor}

\begin{proof}
Fix $P \in$ Spec($R$). Then $A \otimes_R k(P) = k(P)[X_P, Y_P]$ for some $X_P, Y_P \in A \otimes_R k(P)$. Let $\bar{a}$ and $\bar{b}$ respectively be the images of $a$ and $b$ in $A \otimes_R k(P)$. Suppose $\bar{b} \ne 0$ in $A \otimes_R k(P)$. Since $\displaystyle \frac{(A \otimes_R k(P))[T]}{(\bar{b}T^n -\bar{a})} = \frac{k(P)[X_P, Y_P][T]}{(\bar{b}T^n-\bar{a})} = k(P)^{[2]}$, by Theorem \ref{our_field-th}, we get $A \otimes_R k(P) = (R[a] \otimes_R k(P))^{[1]}$ and $A[T] \otimes_R k(P) = (R[bT^n-a, T] \otimes_R k(P))^{[1]}$. Now suppose $\bar{b} = 0$ in $A \otimes_R k(P)$. Then $\displaystyle \frac{(A \otimes_R k(P))[T]}{(\bar{b}T^n -\bar{a})} = \frac{k(P)[X_P, Y_P][T]}{(\bar{b}T^n-\bar{a})} = \left(\frac{k(P)[X_P, Y_P]}{(\bar{a})}\right)^{[1]} = k(P)^{[2]}$ and hence by (\cite{AEH_Coff}, Theorem 3.3) $\displaystyle \frac{k(P)[X_P, Y_P]}{(\bar{a})} = k(P)^{[1]}$. Since the characteristic of $k(P)$ is 0, by Abhyankar-Moh-Suzuki Epimorphism Theorem (\cite{AM_Epi}, \cite{SU_Epi}), we get $A \otimes_R k(P) = k(P)[X_p, Y_P] = k(P)[\bar{a}]^{[1]} = (R[a] \otimes_R k(P))^{[1]}$; and hence $A[T] \otimes_R k(P) = k(P)[\bar{a},T]^{[1]} = k(P)[\bar{b}T^n-\bar{a}, T]^{[1]} = (R[bT^n-a, T] \otimes_R k(P))^{[1]}$.

\medskip

Since $P \in$ Spec($R$) is arbitrary, using Theorem \ref{DD_res_lem-1} we get $A = R[a]^{[1]}=R^{[2]}$ and $A[T] = R[bT^n-a, T]^{[1]}$. This completes the proof.
\end{proof}

The following example by S.M. Bhatwadekar shows that the condition ``$A$ is a normal $k$-domain'' is necessary for Theorem \ref{our_field-th}.

\begin{ex} \label{example-1}
Let $B = k[X,Y] = k^{[2]}$ and $I = (X^2, Y-1)$ be an ideal of $B$. Let $A = k +I$. Then $B$ is a finite  birational extension of $A$ and the conductor of $B$ over $A$  is $I$. Let $F = X^2 T^2 - Y \in A[T]$. 

\medskip

Note that $F$ is a prime element of $B[T]$  such that $B[T]/FB[T] = k[X,T]$. Moreover, $F +1 = X^2T^2 - (Y-1) \in IA[T]$. From this it follows that $A[T] + FB[T] = B[T]$ as $IB[T] = IA[T] $ and $FA[T] = FB[T] \cap A[T]$.

\medskip

Thus $A[T]/FA[T] = B[T]/FB[T] = k[X,T] = k^{[2]}$.
\end{ex}

The next example shows the necessity of the hypothesis ``$b \notin PA_P$ for all $P \in$ Spec($R$)'' in Proposition \ref{our_UFD-normal-th}.
\begin{ex} \label{example-2}
Let $(R, \pi)$ be a DVR and $\displaystyle A = \frac{R[X,Y,Z]}{(\pi^m Z+(X-1)Y-1)}$ where $m \ge 1$. Let $x$, $y$ and $z$ respectively denote the images of $X$, $Y$ and $Z$ in $A$. Set $\displaystyle B := \frac{A[T]}{(\pi y T^n -x)}$ where $n \ge 1$. We claim that $B = A[t] = R[t]^{[1]}$, where $t$ is the image of $T$ in $B$.

\medskip

Note that $\pi \in R$ is prime in both $R[t]$ and $A[t]$, 

$$\displaystyle \frac{A[t]}{\pi A[t]} = \frac{R}{\pi R}[t, z] = \left( \frac{R[t]}{\pi R[t]} \right)^{[1]} ~\, \mbox{and} \, ~~ A[t]\left[\frac{1}{\pi}\right] = R[t]\left[\frac{1}{\pi}\right][y] = \left( R[t]\left[\frac{1}{\pi}\right]\right)^{[1]}.$$ 

Hence, by a version of the Russell-Sathaye criterion (\cite{RS_FIND}, Theorem 2.3.1) for a ring to be a polynomial algebra over a subring (\cite{BD_DVR}, Theorem 2.6) we get $A[t] = R[t]^{[1]} = R^{[2]}$. But $A \ne R^{[2]}$, since $A/\pi A \neq (R/\pi R)^{[2]}$.
\end{ex}

 \begin{rem}
 Let $R$ be a Noetherian domain and $A = R[X,Y]$. Suppose $a, b \in A$ be such that $\displaystyle \frac{A[T]}{(bT-a)} = R^{[2]}$.

\item [\rm (I)]
 When $b \notin PA_P$ for all $P \in $ Spec($R$), then by the contributions of Sathaye (\cite{Sathaye_OLP}, Theorem) and Russell (\cite{R_BIR}, Theorem 2.3), and by a result on residual variables by Das-Dutta (\cite{DD_residual}, Theorem 3.13) it can be seen that $A[T]$ is a stably polynomial algebra over $R[bT-a]$. Further, when $R$ is a Dedekind domain, it is known that $A[T]$, in fact, is a polynomial ring in two variables over $R[bT-a]$ (\cite{BD_DVR} and \cite{DD_plane}, Theorem 3.2); but it is not known whether $A[T] = R[bT-a]^{[2]}$ holds in general.
 
 \item [\rm (II)]
 When $b \in PA_P$ for some $P \in$ Spec($R$), it is not known whether $A[T] = R[bT-a]^{[2]}$ even if $R$ is a DVR. We quote below a concrete example by Bhatwadekar-Dutta (\cite{BD_AFNFIB}, Example 4.13 and \cite{Amartya-Neena_Epi}, Example 4.7). Let $(R,\pi)$ be a DVR containing $\mathbb{Q}$ and $F = \pi Y^2 T + X + \pi X (Y + Y^2) + \pi^2 Y \in A[T]$.  Then it can be seen that $A[T]^{[1]} = R[F]^{[3]}$, $A[T]/(F) = R^{[2]}$, $A[T][1/\pi] = R[1/\pi][Y,F]^{[1]}$ and $ A[T] \otimes_R R/\pi R = (R/\pi R)[\bar{F}]^{[2]}$ where $\bar{F}$ is the image of $F$ in $A[T] \otimes_R R/\pi R$; but it is not known whether $A[T]= R[F]^{[2]}$. 
\end{rem}

\noindent
{\bf Acknowledgement:} The author thanks Amartya K. Dutta for going through the draft carefully and pointing out some mistakes, and also thanks Neena Gupta for her suggestions and for formulation of Example \ref{example-2}.

\bibliographystyle{amsplain}
\bibliography{reference}
\end{document}